\documentclass[12pt]{amsart} 
\usepackage{fancyhdr}
\usepackage{enumitem}
\usepackage{amsmath,amsthm,amssymb,mathrsfs}
\usepackage{tikz-cd}
\usepackage{marginnote}
 
\providecommand{\binom}[2]{{#1\choose#2}}

\renewcommand{\geq}{\geqslant}
\renewcommand{\leq}{\leqslant}
\renewcommand{\H}{\mathrm{H}}                          

\newcommand{\K}{\mathrm{K}}      
\newcommand{\Ish}{\mathcal{I}}
\newcommand{\kk}{\mathbf{k}}
\newcommand{\spec}{\operatorname{Spec}}
\renewcommand{\AA}{\mathbb{A}} 
\newcommand{\QQ}{\mathbb{Q}} 
\newcommand{\RR}{\mathbb{R}} 
\newcommand{\ZZ}{\mathbb{Z}} 
\newcommand{\GG}{\mathbb{G}}

\newtheorem{theorem}{Theorem}[section]

\newtheorem{proposition}[theorem]{Proposition}
\theoremstyle{definition}

\newtheorem{example}[theorem]{Example}

\numberwithin{equation}{section}

\title[Slope $\K$-semistability for big and nef line bundles]{\uppercase{CM}-line bundles and slope $\K$-semistability for big and nef line bundles along subschemes}

\author{Nathan Grieve}

\address{
School of Mathematics and Statistics, Carleton University, 4302 Herzberg Laboratories,  1125 Colonel By Drive, Ottawa, ON, K1S 5B6, Canada; 
D\'{e}partement de math\'{e}matiques, Universit\'{e} du Qu\'{e}bec \`a Montr\'{e}al, Local PK-5151, 201 Avenue du Pr\'{e}sident-Kennedy, Montr\'{e}al, QC, H2X 3Y7, Canada; Department of Pure Mathematics, University of Waterloo, 200 University Avenue West, Waterloo, ON, N2L 3G1, Canada
}
\email{nathan.m.grieve@gmail.com}%

\begin{document}

\begin{abstract}
We apply the theory of the Chow-Mumford line bundle as developed by Arezzo-et-al and build on earlier key insights of Paul and Tian (see \cite{Arezzo:DellaVedova:LaNave} and the references therein).  

In particular, we give an explicit intersection theoretic description of the Donaldson-Futaki $\K$-stability invariants that arise via deformation to the normal cones along subschemes and with respect to big and nef line bundles on projective varieties.  

In doing so we generalize to the case of big and nef line bundles the slope stability theory of Ross and Thomas \cite[Section 4]{Ross:Thomas:2006}.  A key point to what we do here is the continuity property of the Chow Mumford line bundles with respect to the ample cones of projective varieties.
\end{abstract}

\thanks{
The author thanks the Natural Sciences and Engineering Research Council of Canada for their support through his grants DGECR-2021-00218 and RGPIN-2021-03821. \\
\textbf{Article published as: ASIAN J. MATH. Vol. 29, No. 6, pp. 703–710, December 2025.}
}

\maketitle

\section{Introduction}\label{intro}

In recent times, there has been a continued interest in the study of \emph{birational divisors} and their relation to concepts in $\K$-stability.  In the context of the Zariski Riemann spaces associated to Fano varieties, a recent work is \cite{Trusiani:2024}.  A complementary circle of ideas is the non-{A}rchimedean approach to $\K$-stability that has been developed by Boucksom and Jonsson \cite{Boucksom:Jonsson:2023}, \cite{Boucksom:Jonsson:2023a}.

Turning to similar questions for birational $\K$-stability invariants for the Zariski Riemann spaces in general, and with respect to arbitrary and \emph{sufficiently positive} birational line bundles, a first step is to understand what can be done for the case of big and nef line bundles on projective varieties and, in particular, as an extension to the slope stability theory that has been developed by Ross and Thomas \cite{Ross:Thomas:2006}, \cite{Ross:Thomas:2007}.  

Such is the purpose for what we do here.  Especially, we apply the theory of Chow-Mumford (CM) line bundles to introduce a concept of \emph{slope stability along subschemes $Z \subsetneq X$ and with respect to big and nef line bundles} on a projective variety $X$. 

 In more detail, we build on the viewpoint of Arezzo-et-al which, in turn, builds on earlier key insights of Paul and Tian (see \cite{Arezzo:DellaVedova:LaNave} and the references therein).  
 
Working over an algebraically closed characteristic zero base field, we prove the following result.

\begin{theorem}\label{Ross:Thomas:Slope:Stab:DF:intro}
Suppose that $L$ is a big and nef line bundle on an $n$-dimensional normal projective variety $X$.  Suppose that $Z \subsetneq X$ is a proper subscheme and with ideal sheaf $\Ish_Z$.  Denote by $\pi \colon X' \rightarrow X$ the normalized blowing-up morphism of $X$ along $Z$ and having exceptional divisor $E$.  Let $\epsilon(L;Z)$ denote the nef threshold of $X$ with respect to $Z$.  For each $0 \leq t \leq \epsilon(L;Z)$ set
\begin{equation}\label{alpha0}
\alpha_0(L;Z; t) := \frac{((\pi^*L - t E)^n )}{n!}
\end{equation}
\begin{equation}\label{alpha1}
\alpha_1(L;Z; t) := - \frac{(\K_{X'} \cdot(\pi^*L - t E)^{n-1})}{2(n-1)!}
\end{equation}
and
\begin{equation}\label{slope}
\mu(X;L) := \frac{\alpha_1}{\alpha_0} = \frac{\alpha_1(0)}{\alpha_0(0)} \text{.}
\end{equation}
For each 
$c \in (0, \epsilon(L;Z)] \bigcap \QQ$
set
$$
\mu_c(L;\Ish_Z) := \frac{ \int_0^c \left( \alpha_1(L;Z;t) + \frac{\alpha_0'(L;Z;t)}{2} \right) \mathrm{d}t }{ \int_0^c \alpha_0(L;Z;t) \mathrm{d}t } \text{.}
$$
Fix a sufficiently divisible integer $m \in \ZZ$ so that $mc \in \ZZ$.  Then, with this notation and hypothesis the big and nef deformation to the normal cone test configuration $(\mathcal{X}, \mathcal{L}_c^{\otimes m})$ has Donaldson-Futaki invariant $\operatorname{DF}(\mathcal{X},\mathcal{L}_c^{\otimes m})$ equal to some positive scalar multiple of 
$$\mu(X;L) - \mu_c(L;\Ish_Z) \text{.}$$
\end{theorem}

Theorem \ref{Ross:Thomas:Slope:Stab:DF:intro} extends the theory of Ross and Thomas \cite[Section 4]{Ross:Thomas:2007}.  We prove it in Section \ref{extended:ross:thomas}.  A key point is the continuity property of the CM-line bundles on the ample cone (see Theorem \ref{CM:line:bundle:epsilon:zero}).  This builds on earlier work of Arezzo-et-al, Paul and Tian (see \cite{Arezzo:DellaVedova:LaNave} and the articles cited there).  We refer to Section \ref{extended:ross:thomas} for a detailed discussion of those test configurations which arise via the deformation to the normal cone.

While in many ways, Theorem \ref{Ross:Thomas:Slope:Stab:DF:intro} should be seen as a non-surprising result, to help illustrate its depth it is worthwhile to mention that---independent of the theory of the CM-line bundles---it is possible to define such slope stability invariants, along subschemes, at the level of numerical classes of divisors and \`a priori with no relation to test configurations.  

However, in a more conceptual sense, what Theorem \ref{Ross:Thomas:Slope:Stab:DF:intro} is saying is that such numerically defined invariants can be connected in a meaningful way via the theory of CM-line bundles to the classical context of test configurations.  This is a significant and aesthetically pleasing point which is a result that is of an independent interest.  

\subsection*{Conventions and other notation}  In this article, $\mathbf{k}$ denotes an algebraically closed characteristic zero base field.  By a \emph{$\kk$-variety}, or  simply \emph{variety} over $\kk$, we mean an integral finite type $\kk$-scheme.  In what follows 
$$\AA^1 := \spec \kk[t]$$ 
denotes the affine line over $\kk$ and 
$$\GG_m := \spec \kk[t,t^{-1}]$$ 
denotes the multiplicative group of nonzero elements of $\kk$.

\subsection*{Acknowledgements}
The author thanks the Natural Sciences and Engineering Research Council of Canada for their support through his grants DGECR-2021-00218 and RGPIN-2021-03821. 
It is the author's pleasure to thank colleagues for their interest and discussions on related topics.

\section{Preliminaries about test configurations}

Our approach to test configurations is similar to that of \cite{Boucksom-Hisamoto-Jonsson:2016}.  Specifically, we consider the case a projective variety $X$ together with a $\QQ$-line bundle $L$ on $X$.  Then a \emph{test configuration} $(\mathcal{X},\mathcal{L})$ for $X$ with respect to $L$ consists of a $\GG_m$-equivariant flat and proper morphism
\begin{equation}\label{test:config:eqn:1}
\pi \colon \mathcal{X} \rightarrow \AA^1
\end{equation}
which has the property that 
\begin{equation}\label{test:config:eqn:2}
\mathcal{X}\big|_{\pi^{-1}(t)} \simeq X \text{ for $t \in \AA^1 \setminus \{0\}$}
\end{equation}
together with a $\GG_m$-linearized $\QQ$-line bundle $\mathcal{L}$ on $\mathcal{X}$ and an isomorphism
\begin{equation}\label{test:config:eqn:3}
\mathcal{L}\big|_{\pi^{-1}(t)} \simeq L  \text{ for $t \in \AA^1 \setminus \{0\}$.}
\end{equation}
The isomorphisms of $\QQ$-line bundles \eqref{test:config:eqn:3} are required to be compatible with the isomorphism \eqref{test:config:eqn:2}.  The $\GG_m$-action on $\mathcal{X}$ and $\mathcal{L}$ respectively \emph{cover} the canonical action of $\GG_m$ on $\AA^1$.

As in \cite{Mum:GIT}, our conventions for $\GG_m$-actions are such that if $\GG_m$ acts on the total space of $\mathcal{L}$ by $t^\lambda$, for $\lambda \in \ZZ$, then the linearized action on $\mathcal{L}$ has the property that $\GG_m$ acts by $t^{-\lambda}$.   

\begin{example}
For the case that $L$ and $\mathcal{L}$ are respectively ample and $\pi$-ample then we say that $(\mathcal{L},\mathcal{X})$ is an \emph{ample test configuration}.  More generally, if $L$ is big, respectively nef, then we say that $(\mathcal{X},\mathcal{L})$ is a \emph{big}, respectively \emph{nef}, test configuration if $\mathcal{L}$ is big, respectively nef, over $X$.
  Note, in particular, that our conventions for $\GG_m$-actions are as in \cite{Boucksom-Hisamoto-Jonsson:2016}.  So, in particular they are such that for ample test configurations $(\mathcal{X},\mathcal{L})$ the \emph{Donaldson-Futaki invariant} $\operatorname{DF}(\mathcal{X},\mathcal{L})$ arises in the asymptotic expansion as
\begin{equation}\label{DF:expansion:eqn}
\begin{split}
\frac{w_m}{m h^0(X,mL)} & = \frac{1}{m h^0(X,mL)}\sum_{\lambda \in \ZZ} \lambda \dim \H^0(X,mL)_{\lambda} \\
&= \int_{\RR} \lambda \operatorname{DH}_{(X,L)} \mathrm{d}\lambda - \frac{\operatorname{DF}(\mathcal{X},\mathcal{L})}{m} + \mathrm{O}(m^{-2}) \text{.}
\end{split}
\end{equation}
Here, $\H^0(X,mL)_{\lambda}$ is the $\lambda$-weight space of the induced $\GG_m$-action of $\H^0(X,mL)$ whereas $w_m$ is the weight of the induced action of $\GG_m$ on $\operatorname{det} \H^0(X,mL)$.  

The quantity 
$$
\int_{\RR} \lambda \operatorname{DH}_{(X,L)} \mathrm{d}\lambda
$$
admits a description as the expectation of the {D}uistermaat-{H}eckman measures.  It can be calculated via the theory of concave transforms for Newton-Okounkov bodies (see for example \cite{Grieve:MVT:2019}, \cite{Grieve:chow:approx} and the references therein).  In \eqref{DF:expansion:eqn}, and elsewhere in this article, we note that our conventions for $\operatorname{DF}(\mathcal{X},\mathcal{L})$ differ by a positive scalar multiple from those of \cite{Boucksom-Hisamoto-Jonsson:2016}.  

By our conventions, which are mostly consistent with those of \cite{Boucksom-Hisamoto-Jonsson:2016}, a polarized projective variety $(X,L)$ is \emph{$\K$-semistable} if 
$\operatorname{DF}(\mathcal{X},\mathcal{L}) \geq 0$ for all ample test configurations $(\mathcal{X},\mathcal{L})$.
A polarized projective variety $(X,L)$ is \emph{$\K$-stable} if it is $\K$-semistable and $\operatorname{DF}(\mathcal{X},\mathcal{L}) = 0$ exactly when $(\mathcal{X},\mathcal{L})$ is \emph{almost trivial} in the sense that if $\mathcal{X}'$ is the normalization of $\mathcal{X}$ and if $\mathcal{L}'$ is the pull-back of $\mathcal{L}$ to $\mathcal{X}'$, then 
$\mathcal{X}' \simeq X_{\AA^1} := X \times \AA^1$, $\mathcal{L}' \simeq  L_{\AA^1}$ and $\GG_m$ acts by a character. 
\end{example}

\section{Preliminaries about CM-line bundles}

We recall a construction of CM-line bundles following the approach from \cite{Arezzo:DellaVedova:LaNave}.  In particular, let $f \colon X \rightarrow B$ be a flat family of $n$-dimensional projective schemes with $n \geq 1$.  Let $L$ be a line bundle on $X$ and for $m \geq 0$, let $\mathbf{R}f_*(L^{\otimes m})$ be the derived pushforward of $L^{\otimes m}$.  Here, we view $L^{\otimes m}$ as a complex of perfect sheaves on $X$ and supported in cohomological degree zero.  In particular, the rank of $\mathbf{R}f_*(L^{\otimes m})$ equals the Euler characteristic of $L^{\otimes m}$ restricted to any, and all, fibres of $f$.

In this context, there exists a polynomial expansion
$$
\operatorname{rank} \mathbf{R}f_*(L^{\otimes m}) = a_0 m^n + a_1 m^{n-1} + \dots + a_n
$$
with $a_i \in \QQ$.  Further, there are line bundles $\nu_i$ on $B$ and depending on $f$ and $L$ which are such that 
$$
\operatorname{det} \mathbf{R}f_*(L^{\otimes m}) = \nu_0^{\otimes \binom{m}{n+1} } \otimes \nu_1^{\otimes \binom{m}{n}} \otimes \dots \otimes \nu_{n+1} \text{.}
$$
Put differently, there exist $\QQ$-line bundles $\mu_i$ on $B$ which are such that 
$$
\operatorname{det} \mathbf{R} f_*(L^{\otimes m}) = \mu_0^{\otimes m^{n+1}} \otimes \mu_1^{\otimes m^n} \otimes \dots \otimes \mu_{n+1} \text{.}
$$
So, in particular, we may consider the following asymptotic expansion
$$
\operatorname{det} \mathbf{R}f_* (L^{\otimes m})^{\frac{1}{\operatorname{rank} \mathbf{R} f_*(L^{\otimes m})}} = \mu_0^{\otimes \frac{m}{a_0}} \otimes (\mu_1^{a_0} \otimes \mu_0^{-a_1})^{\otimes \frac{1}{a_0^2}} \otimes \mathrm{O}\left(\frac{1}{m}\right) 
$$
as $m \to \infty$.

In this context, the CM-line bundle $\lambda_{\mathrm{CM}}(X,L)$ is defined to be the $\QQ$-line bundle on $B$ that is defined by the condition that
$$
\lambda_{\mathrm{CM}}(X,L) = \left( \mu_1^{\otimes a_0} \otimes \mu_0^{\otimes -a_1} \right)^{\otimes \frac{1}{a_0^2}} \text{.}
$$
For later use, we record the following two properties of CM-line bundles.

\begin{proposition}\label{CM:line:bundle:first:properties}
Let $f \colon X \rightarrow B$ be a flat family of $n$-dimensional projective schemes.  Suppose that $L$ is an $f$-relative big and nef line bundle on $X$.   Then, the following assertions hold true.
\begin{itemize}
\item[(i)]{
$\lambda_{\mathrm{CM}}(X,L^{\otimes m}) = \lambda_{\mathrm{CM}}(X,L)$ for $m > 0$.
}
\item[(ii)]{
Suppose that $A$ is an $f$-ample line bundle on $X$.  Then
$$
\lambda_{\mathrm{CM}}(X,L^{\otimes m} \otimes A) = \lambda_{\mathrm{CM}}(X,L) \otimes \mathrm{O} \left( \frac{1}{m} \right) \text{ as $m \to \infty$.}
$$
}
\end{itemize}
\end{proposition}
\begin{proof}
Item (i) is \cite[Proposition 2.7 (1)]{Arezzo:DellaVedova:LaNave}.  Item (ii) is \cite[Proposition 2.8]{Arezzo:DellaVedova:LaNave} and is a consequence of \cite[Proposition 2.7 (1) \& (2)]{Arezzo:DellaVedova:LaNave}.
\end{proof}

Combining Kleiman's theorem about numerical characterization of ample divisors, see for example \cite[Corollary 1.4.10]{Laz}, with Proposition \ref{CM:line:bundle:first:properties} gives rise to Theorem \ref{CM:line:bundle:epsilon:zero} below.  Its essential content, in particular, is that in calculating DF-invariants for big and nef line bundles we can pass to the limit $\epsilon \to 0$.  

\begin{theorem}\label{CM:line:bundle:epsilon:zero}
Let $f \colon X \rightarrow B$ be a flat family of $n$-dimensional projective schemes.  Suppose that $L$ is a big and nef line bundle on $X$.  Fix an ample line bundle $A$ on $X$.  Then
$$\lambda_{\mathrm{CM}}(X,L) = \lim\limits_{\epsilon \to 0} \lambda_{\mathrm{CM}}(X,L \otimes A^{\otimes \epsilon})  \text{.} $$
\end{theorem}
\begin{proof}
It suffices to consider sufficiently small rational numbers 
$$\epsilon = \frac{1}{m} > 0$$ 
for sufficiently large positive integers $m$.  Indeed, by Kleiman's theorem, \cite[Corollary 1.4.10]{Laz}, for all such positive integers $m$, the class $L \otimes A^{\otimes \frac{1}{m}}$ is ample.  

On the other hand, we can also write 
$$\left(L \otimes A^{\otimes \frac{1}{m}} \right)^{\otimes m} = L^{\otimes m} \otimes A \text{.}$$  
Thus, the desired conclusion follows as a consequence of Proposition \ref{CM:line:bundle:first:properties} (i) and (ii).
\end{proof}

\section{Calculation of DF-invariants via the theory of CM-line bundles}

Consider now the case of a test configuration $f \colon (\mathcal{X}, \mathcal{L}) \rightarrow \AA^1$ for a big and nef line bundle $L$ on a projective variety $X$.  So, in particular, the $\mathbb{G}_m$-action on $X$ lifts to a linearized action of $\mathbb{G}_m$ on $L$.

In this context, the CM-line bundle $\lambda_{\mathrm{CM}}(X,L)$ admits a natural $\mathbb{G}_m$-linearization.  Further, as noted by Paul and Tian, compare also with the discussion given in \cite[Remark 2.3]{Arezzo:DellaVedova:LaNave}, the weight of the $\mathbb{G}_m$-action on the fibre  $\lambda_{\mathrm{CM}}(X,L)|_{t = 0}$ equals the Donaldson-Futaki invariant $\operatorname{DF}(X,L)$ of the test configuration $(\mathcal{X},L)$.  Here, and elsewhere our conventions for $\mathbb{G}_m$-actions and Donaldson-Futaki invariants are such that their non-negativity for all test configurations characterizes $\K$-semistability for $X$ with respect to $L$.  In this regard, our conventions differ by a minus sign from those of \cite{Arezzo:DellaVedova:LaNave}.

For later use, we record this important insight, of Paul and Tian, in Theorem \ref{Paul:Tian:DF:Thm} below.

\begin{theorem}[Paul and Tian; Arezzo-et-al {\cite[Proposition 3.3]{Arezzo:DellaVedova:LaNave}}]\label{Paul:Tian:DF:Thm}
Let $(\mathcal{X},\mathcal{L})$ be a big and nef test configuration for $L$ a big and nef line bundle on a projective variety $X$.  Then the Donaldson-Futaki invariant $\operatorname{DF}(\mathcal{X},\mathcal{L})$ equals the weight of the induced $\mathbb{G}_m$-action on the fibre $\lambda_{\mathrm{CM}}(X,L)|_{t = 0}$ for $\lambda_{\mathrm{CM}}(X,L)$ the CM-line bundle that is determined by $(\mathcal{X},\mathcal{L})$.
\end{theorem}

\begin{proof}
See \cite[Proposition 3.3]{Arezzo:DellaVedova:LaNave}.
\end{proof}

As an illustration of Proposition \ref{CM:line:bundle:first:properties}  and Theorem \ref{Paul:Tian:DF:Thm}, we want to explain how they allow for a description of the Donaldson-Futaki invariants, for big and nef line bundles, that arise via the deformation to the normal cone.  Recall, that those test configurations that arise via the deformation to the normal cone are a central feature to the Ross-Thomas theory of slope stability (\cite{Ross:Thomas:2007} and \cite{Ross:Thomas:2006}).  We discuss this in detail in Section \ref{extended:ross:thomas}.

\section{Slope stability along subschemes with respect to the numerical class of a big and nef line bundle}\label{extended:ross:thomas}

Let $Z \subsetneq X$ be a subscheme of a normal projective variety $X$ and let
\begin{equation}\label{def:normal:cone}
\mathcal{X} = \operatorname{DefNC}(X,Z) := \operatorname{Bl}_{Z \times \{0\}}(X_{\AA^1}) \rightarrow X_{\AA^1}
\end{equation}
be the normalized blowing-up of 
$$
X_{\AA^1} := X \times \AA^1
$$
along $Z \times \{0\}$.  As in \cite[Example, 2.4]{Boucksom-Hisamoto-Jonsson:2016}, see also \cite[Section 4]{Ross:Thomas:2007}, we say that $\mathcal{X}$ is the \emph{deformation to the normal cone} of $X$ with respect to $Z$.

Let $\mathcal{X}_0$ be the fiber of $\mathcal{X}$ over $t = 0 \in \AA^1$ with respect to the naturally induced map $\mathcal{X} \rightarrow \AA^1$.  Then 
$$\mathcal{X}_0 \subseteq \mathcal{X}$$ 
is a divisor and is described as 
$$\mathcal{X}_0 = \mathcal{E} + \mathcal{F} \text{.}$$  
Here $\mathcal{E}$ is the exceptional divisor of the normalized blowing-up morphism \eqref{def:normal:cone} and
$$
\mathcal{F} \simeq \operatorname{Bl}_Z(X)
$$
for
$$
\pi \colon X' :=  \operatorname{Bl}_Z(X) \rightarrow X
$$
the normalized blowing-up of $X$ along $Z$ is the strict transform of $X \times \{0\}$ with respect to the natural  map \eqref{def:normal:cone}.

Let $L$ be a big and nef line bundle on $X$.  Denote by
$$
\epsilon(L;Z) := \sup \{ t \in \RR_{\geq 0} : \pi^* L - t E \text{ is nef} \}
$$
its \emph{nef-threshold} (\emph{Seshadri constant}) with respect to $Z$.

For rational numbers 
$$c \in (0,\epsilon(L;Z)] \bigcap \QQ$$ 
and sufficiently divisible integers $m$ so that $mc \in \ZZ$, denote by $\mathcal{L}^{\otimes m}_c$ the line bundle on $\mathcal{X}$ that is defined by the condition that 
$$
\mathcal{L}_c^{\otimes m} = p^* L^{\otimes m} - c m \mathcal{E} \text{.}
$$
Here 
$p \colon \mathcal{X} \rightarrow X$
denotes the natural map.

Then for each such 
$$c \in (0,\epsilon(L;Z)] \bigcap \QQ$$ 
and each such sufficiently divisible integer $m$, $(\mathcal{X},\mathcal{L}_c^{\otimes m})$ is a big and nef test configuration of $X$ with respect to the big and nef line bundle $L$.

We want to apply the theory of CM-line bundles to describe the DF-invariant of such test configurations $(\mathcal{X}, \mathcal{L}^{\otimes m}_c)$ in terms of those that are constructed from small ample perturbations of $L$.  Such considerations give rise to the following more general form of the very useful and explicit Ross-Thomas description of the Donaldson-Futaki invariants for the case of deformation to the normal cone test configurations with respect to ample line bundles.  

Indeed, Theorem \ref{Ross:Thomas:Slope:Stab:DF} below, which we also stated as Theorem \ref{Ross:Thomas:Slope:Stab:DF:intro} in Section \ref{intro}, applies the theory of CM-line bundles, in the form of Corollary \ref{CM:line:bundle:epsilon:zero}.  It is consequence of the work of Ross and Thomas \cite[Section 4]{Ross:Thomas:2007} combined with the Riemann-Roch Theorem for normal varieties (e.g., in the form of \cite[Theorem A.1]{Boucksom-Hisamoto-Jonsson:2016}).  

\begin{theorem}\label{Ross:Thomas:Slope:Stab:DF}
Suppose that $L$ is a big and nef line bundle on an $n$-dimensional normal projective variety $X$.  Suppose that $Z \subsetneq X$ is a proper subscheme and with ideal sheaf $\Ish_Z$.  Let $\epsilon(L;Z)$ denote the nef threshold of $X$ with respect to $Z$.  For each $0 \leq t \leq \epsilon(L;Z)$ set
\begin{equation}\label{alpha0}
\alpha_0(L;Z; t) := \frac{((\pi^*L - t E)^n )}{n!}
\end{equation}
\begin{equation}\label{alpha1}
\alpha_1(L;Z; t) := - \frac{(\K_{X'} \cdot(\pi^*L - t E)^{n-1})}{2(n-1)!}
\end{equation}
and
\begin{equation}\label{slope}
\mu(X;L) := \frac{\alpha_1}{\alpha_0} = \frac{\alpha_1(0)}{\alpha_0(0)} \text{.}
\end{equation}
For each 
$c \in (0, \epsilon(L;Z)] \bigcap \QQ$
set
$$
\mu_c(L;\Ish_Z) := \frac{ \int_0^c \left( \alpha_1(L;Z;t) + \frac{\alpha_0'(L;Z;t)}{2} \right) \mathrm{d}t }{ \int_0^c \alpha_0(L;Z;t) \mathrm{d}t } \text{.}
$$
Fix a sufficiently divisible integer $m \in \ZZ$ so that $mc \in \ZZ$.  Then, with this notation and hypothesis the test configuration $(\mathcal{X}, \mathcal{L}_c^{\otimes m})$ has Donaldson-Futaki invariant $\operatorname{DF}(\mathcal{X},\mathcal{L}_c^{\otimes m})$ equal to some positive scalar multiple of $\mu(X;L) - \mu_c(L;\Ish_Z)$.
\end{theorem}

\begin{proof}
The idea is to use Corollary \ref{CM:line:bundle:epsilon:zero} together with the Riemann-Roch Theorem for normal varieties (e.g., in the form of \cite[Theorem A.1]{Boucksom-Hisamoto-Jonsson:2016}) to reduce the proof of the theorem to the classical case of Ross and Thomas \cite[Section 4]{Ross:Thomas:2007}.  

Fix a sufficiently small rational number $\epsilon > 0$ and consider the ample line bundle $L + \epsilon H$.  Suppose that 
$$c \in (0,\epsilon(L;Z)] \bigcap \QQ \text{.}$$  
Then $\pi^*L - c E$ is nef.  Also, 
$$\pi^*L \leq \pi^*(L+ \epsilon H)$$ 
whence 
$$0 \leq \pi^*L - c E \leq \pi^*(L+\epsilon H) - c E\text{.}$$  
Thus $\pi^*(L+\epsilon H) - c E$ is nef too.    In other words
$$
c \in (0,\epsilon(L+\epsilon H;Z)] \bigcap \QQ \text{.}
$$

Let $\mathcal{X} := \operatorname{DefNC}(X,Z)$ denote the deformation to the normal cone.   Replacing $m$ by some suitably large divisible integer $m$ if necessary, consider the test configurations 
$$(\mathcal{X}, p^*L^{\otimes m} - c m \mathcal{E})$$ 
and
$$(\mathcal{X}, p^*(L^{\otimes  m}  \otimes H^{\otimes \epsilon m}) - cm \mathcal{E}) ) \text{.}$$

Combining Corollary \ref{CM:line:bundle:epsilon:zero} and Theorem \ref{Paul:Tian:DF:Thm} it follows that
$$
\operatorname{DF}( (\mathcal{X}, p^*L^{\otimes m} - c m \mathcal{E}) ) = \lim\limits_{\epsilon \to 0} \operatorname{DF}( (\mathcal{X}, p^*(L^{\otimes m}  \otimes H^{\otimes \epsilon m}) - cm \mathcal{E})  ) 
$$

On other other hand, the conclusion of Theorem \ref{Ross:Thomas:Slope:Stab:DF} applied to $(\mathcal{X}, p^*(L^{\otimes  m}  \otimes H^{\otimes \epsilon m}) - cm \mathcal{E})  )$ follows as a consequence of the theory of Ross and Thomas \cite[Section 4]{Ross:Thomas:2007} combined  with the Riemann-Roch Theorem for normal varieties (e.g., in the form of \cite[Theorem A.1]{Boucksom-Hisamoto-Jonsson:2016}).

Finally, since
$$
\mu(X;L) - \mu_c(L;\Ish_Z) = \lim\limits_{\epsilon \to 0} (\mu(X;L+\epsilon H) - \mu_c(L+\epsilon H;\Ish_Z) )
$$
the conclusion desired by Theorem \ref{Ross:Thomas:Slope:Stab:DF} then follows.
\end{proof}

\providecommand{\bysame}{\leavevmode\hbox to3em{\hrulefill}\thinspace}
\providecommand{\MR}{\relax\ifhmode\unskip\space\fi MR }
\providecommand{\MRhref}[2]{%
  \href{http://www.ams.org/mathscinet-getitem?mr=#1}{#2}
}
\providecommand{\href}[2]{#2}


\end{document}